\documentclass[a4paper,12pt]{article}     
\usepackage{amsmath,amscd,amssymb}        
\usepackage{latexsym}                     
\usepackage[english, german]{babel}                

\usepackage{theorem}                 

\newtheorem{theorem}{Theorem}
\newtheorem{corollary}{Corollary}
\newtheorem{proposition}{Proposition}
\newtheorem{lemma}{Lemma}

\newenvironment{definition}
{\smallskip\noindent{\bf Definition\/}:}{\smallskip\par}

\newenvironment{remark}
{\smallskip\noindent{\bf Remark\/}.}{\smallskip\par}
\newenvironment{remarks}
{\smallskip\noindent{\bf Remarks\/}.}{\smallskip\par}

\newenvironment{proof}{\begin{ProofwCaption}{Proof}}{\end{ProofwCaption}}
\newenvironment{proof*}[1]{\begin{ProofwCaption}{{#1}}}{\end{ProofwCaption}}
\newenvironment{ProofwCaption}[1]%
  {\addvspace\theorempreskipamount \noindent{\it #1.}\rm}%
  {\qed \par \addvspace\theorempostskipamount}
\newcommand{\qedsymbol}{\mbox{$\Box$}}
\newcommand{\qed}{\hfill\qedsymbol}

\newcommand{\CC}{{\mathbb C}}

\newcommand{\ZZ}{{\mathbb Z}}

\newcommand{\Kring}[1]{K_0(\mbox{f.}{#1}\,\mbox{-sets})}

\title{Saito duality between Burnside rings for invertible polynomials}
\author{Wolfgang Ebeling and Sabir M.~Gusein-Zade
\thanks{Partially supported by the DFG Mercator program (INST 187/490-1), the Russian government grant 11.G34.31.0005, RFBR--10-01-00678,
NSh--8462.2010.1.
AMS 2010 Math. Subject Classification: 14J33, 32S40, 19A22.
}
}
\date{}

\begin{document}
\selectlanguage{english}

\maketitle

\begin{abstract}
We give an equivariant version of the Saito duality which can be
regarded as a Fourier transformation on Burnside rings. We show that
(appropriately defined) reduced equivariant monodromy zeta functions of Berglund--H\"ubsch dual invertible
polynomials are Saito dual to each other with respect to their groups of diagonal symmetries. Moreover we show that the relation between ``geometric roots'' of the
monodromy zeta functions for some pairs of Berglund--H\"ubsch dual invertible polynomials described in
a previous paper is a particular case of this duality.
\end{abstract}

\section*{Introduction} 
In a number of papers it was shown that the Poincar\'e series of some natural filtrations on the rings
of germs of functions on singularities are related (sometimes coincide) with appropriate monodromy
zeta functions. In some cases (see, e.g., 
\cite{MRL}) this relation is described in terms
of the so-called Saito duality (\cite{Saito1}, \cite{Saito2}).
This duality also participates in relations between
monodromy zeta functions of Berglund--H\"ubsch dual invertible polynomials: \cite{ET1}, \cite{MMJ}.
Non-degenerate invertible polynomials are the potentials of invertible Landau--Ginzburg models in string theory: \cite{Kreuzer}.
Berglund--H\"ubsch dual non-degenerate invertible polynomials are particular cases of the homological mirror
symmetry for hypersurface singularities: \cite{ET1}. In \cite{ET2}, it was shown that this symmetry
can be extended to orbifold Landau--Ginzburg models, i.e.\ to pairs $(f,G)$ consisting of an invertible polynomial $f$ and a certain abelian group
$G$ of its symmetries. This gives a hint that there can exist an equivariant version of the Saito
duality which participates in relations between monodromies of dual invertible polynomials with
fixed symmetry groups. 
Here we give an equivariant version of the Saito duality which can be
interpreted as a Fourier transformation on Burnside rings. This gives a better understanding of the Saito duality. We show that
(appropriately defined) reduced equivariant monodromy zeta functions of Berglund--H\"ubsch dual invertible
polynomials are Saito dual to each other with respect to their groups of diagonal symmetries. Moreover we show that the relation between ``geometric roots'' of the reduced
monodromy zeta functions for some pairs of Berglund--H\"ubsch dual invertible polynomials described in
\cite{MMJ} is a particular case of this duality.

Saito duality is a duality between rational functions of the form 
\begin{equation}\label{phi}
\varphi(t)=\prod\limits_{m|d}(1-t^m)^{s_m}
\end{equation}
with a fixed positive integer $d$. The Saito dual of $\varphi$ with respect to $d$ is
\begin{equation}\label{Saito}
\varphi^\ast(t) = \prod\limits_{m|d}(1-t^{d/m})^{-s_m}.
\end{equation}

For example, the characteristic polynomials of the classical monodromy operators (or, equivalently, the reduced monodromy zeta functions) of the dual (in the sense of Arnold's strange duality)
pairs of the 14 exceptional unimodular singularities in three variables are Saito dual to each other
with $d$ being the quasidegree of their quasihomogeneous representatives.

The reason for the minus sign in the exponent in the classical
definition of the Saito dual (\ref{Saito}) is connected with
the fact that initially it was applied only to surface
singularities. One can say that, for hypersurface
singularities in $\CC^n$, one should define the Saito dual of
$\varphi$ as $\prod\limits_{m|d}(1-t^{d/m})^{(-1)^{n}s_m}$
(or to say that the reduced monodromy zeta function of an exceptional unimodular singularity with the number of variables different from $3$
is either Saito dual to the reduced monodromy zeta function of its
dual counterpart or is inverse to the dual one). This is the reason why we keep the definition (\ref{Saito}) for rational
functions but shall not follow the sign convention in the
definition of the equivariant version of the Saito duality
below.

\section{Symmetries of invertible polynomials}\label{symmetries}
A quasihomogeneous polynomial $f$ in $n$ variables is called
{\em invertible} (see \cite{Kreuzer}) if it is of the form
\begin{equation}\label{inv}
f(x_1, \ldots, x_n)=\sum\limits_{i=1}^n a_i \prod\limits_{j=1}^n x_j^{E_{ij}}
\end{equation}
for some coefficients $a_i\in\CC^\ast$ and for a matrix
$E=(E_{ij})$ with non-negative integer entries and with $\det E\ne 0$.
Without loss of generality one may assume that $a_i=1$
for $i=1, \ldots, n$. (This can be achieved by a rescaling of the variables $x_j$.)
An invertible quasihomogeneous polynomial $f$ is {\em non-degenerate} if it has (at most) an isolated critical point at the origin in $\CC^n$.

According to \cite{KS}, an invertible polynomial $f$ is non-degenerate if and only if it is a (Thom-Sebastiani) sum of invertible polynomials in groups of different variables of the following types:
\begin{enumerate}
\item[1)] $x_1^{p_1}x_2 + x_2^{p_2}x_3 + \ldots + x_{m-1}^{p_{m-1}}x_m + x_m^{p_m}x_1$ (loop type; $m\ge 2$);
\item[2)] $x_1^{p_1}x_2 + x_2^{p_2}x_3 + \ldots + x_{m-1}^{p_{m-1}}x_m + x_m^{p_m}$ (chain type; $m\ge 1$).
\end{enumerate}

\begin{sloppypar}

An invertible polynomial (\ref{inv}) has a canonical system
of weights ${\bf w}=(w_1, \ldots, w_n;d_f)$, where $w_i$ is the
determinant of the matrix $E$ with the $i$th column
substituted by $(1, \ldots, 1)^T$ and $d_f=\det E$. One has
$f(\lambda^{w_1}x_1, \ldots, \lambda^{w_n}x_n) = \lambda^d f(x_1, \ldots, x_n)$.
The canonical system of weights may be non-reduced, i.e.\ one may have $c_f=\gcd(w_1, \ldots, w_n)\ne 1$. The reduced system
of weights is $\overline{\bf w}=(\overline{w}_1, \ldots, 
\overline{w}_n; \overline{d}_f)=(w_1/c_f, \ldots, w_n/c_f; d_f/c_f)$.

\end{sloppypar}

The {\em Berglund-H\"ubsch transpose} $\widetilde{f}$ of the invertible polynomial (\ref{inv}) is defined by
$$
\widetilde{f}(x_1, \ldots, x_n)=\sum\limits_{i=1}^n a_i \prod\limits_{j=1}^n x_j^{E_{ji}}\,.
$$
If the invertible polynomial $f$ is non-degenerate, then $\widetilde{f}$ is non-degenerate as well.

If the canonical system of weights of $f$ is reduced, the canonical system of weights of $\widetilde{f}$ can be non-reduced
(see examples in \cite{MMJ}).
The canonical quasidegrees of $f$ and $\widetilde{f}$ coincide.

Let $f$ be a quasihomogeneous polynomial in $n$ variables.

\begin{definition}
The (diagonal) {\em symmetry group} of $f$ is the group
$$
G_f=\{(\lambda_1, \ldots, \lambda_n)\in (\CC^*)^n:
f(\lambda_1 x_1, \ldots, \lambda_n x_n)= f(x_1, \ldots, x_n)\}\,,
$$ 
i.e. the group of diagonal linear transformations of $\CC^n$ preserving $f$.
\end{definition}

For an invertible polynomial $f=\sum\limits_{i=1}^n \prod\limits_{j=1}^n x_j^{E_{ij}}$
the symmetry group $G_f$ is finite and is generated by the elements
$$
\sigma_j=(\exp(2\pi i\cdot a_{1j}), \ldots, \exp(2\pi i\cdot a_{nj}))
$$
corresponding to the columns of
the matrix $E^{-1}=(a_{kj})$ inverse to the matrix $E$ of the exponents of $f$. This implies the
following statement

\begin{proposition} {\rm (\cite{Kreuzer}).}
$\vert G_f\vert=d_f$.
\end{proposition}

For an invertible polynomial of loop or chain type the group $G_f$ is a cyclic group of order $d_f$: \cite{Kreuzer}.
For the Thom-Sebastiani sum of polynomials the symmetry group $G_f$ is the direct sum of the corresponding groups for the summands. 

For a finite abelian group $G$, let $G^*=\mbox{Hom\,}(G,\CC^*)$ be its group of characters.
(As abelian groups $G$ and $G^*$ are isomorphic, but not in a canonical way.)

\begin{proposition} {\rm (\cite{BH2}).}
$G_{\widetilde{f}}\cong G_f^*$.
\end{proposition}

One has the following identification of $G_{\widetilde{f}}$ and $G_f^*$. Let
\begin{eqnarray*}
\underline{\lambda} & = & (\exp(2\pi i\,\alpha_1), \ldots, \exp(2\pi i\,\alpha_n))\in G_{\widetilde{f}},\\ 
\underline{\mu} & = & (\exp(2\pi i\,\beta_1), \ldots, \exp(2\pi i\,\beta_n))\in G_f.
\end{eqnarray*}
Define $\langle\underline{\lambda}, \underline{\mu}\rangle_E$ as 
$\langle\underline{\lambda}, \underline{\mu}\rangle_E=\exp(2\pi i\,(\underline{\alpha},\underline{\beta})_E)$,
where
$$(\underline{\alpha},\underline{\beta})_E:=(\alpha_1, \ldots, \alpha_n)E(\beta_1, \ldots, \beta_n)^T.$$
An element $\underline{\mu} =  (\exp(2\pi i\,\beta_1), \ldots, \exp(2\pi i\,\beta_n))$ belongs to $G_f$ if and only if $E(\beta_1, \ldots, \beta_n)^T$ is an integral column vector. This implies that adding integers to $\alpha_1, \ldots, \alpha_n$ does not change $\langle\underline{\lambda}, \underline{\mu}\rangle_E$. The same argument shows that the ambiguity in the choice of $\beta_1, \ldots, \beta_n$ does not influence the value $\langle\underline{\lambda}, \underline{\mu}\rangle_E$. Therefore the pairing $\langle  -, - \rangle_E$ is well defined.
The pairing $\langle  -, - \rangle_E$ associates to an element of $G_{\widetilde{f}}$ a homomorphism $G_f \to \CC^*$. 
Let  $\underline{\lambda} \in G_{\widetilde{f}}$ be such that $\langle\underline{\lambda}, \underline{\mu}\rangle_E=1$ for all $\underline{\mu} \in G_f$. This means that $(\alpha_1, \ldots, \alpha_n)(k_1, \ldots, k_n)^T$ is an integer for arbitrary integers $k_1, \ldots , k_n$ and therefore $\underline{\lambda}=1$. Thus the pairing  defines an isomorphism
between $G_{\widetilde{f}}$ and $G_f^*$. This permits to identify these groups.

The following definition was given in \cite{BH2}.

\begin{definition} (\cite{BH2})
For a subgroup $H\subset G$ its {\em dual} (with respect to $G$) $\widetilde{H}\subset G^*$ is the kernel of the natural map $i^*:G^*\to H^*$ induced by the inclusion $i:H\hookrightarrow G$.
\end{definition}
In \cite[Proposition 3]{ET2}, it was shown that, for $G$ being the symmetry group $G_f$ of a non-degenerate invertible polynomial $f$, this definition coincides with the one from \cite{Krawitz}. 
One can see that the dual to $\widetilde{H}$ (with respect to $G^*$) coincides with $H$,
the dual to $G$ (as a subgroup of $G$ itself) is the trivial subgroup $\langle e\rangle\subset G^*$,
the dual to $\langle e\rangle\subset G$ is the group $G^*$.

\section{Equivariant monodromy zeta function}
Let $G$ be a finite group. A $G$-set is a set with an action of the group $G$. A $G$-set is {\em irreducible} if the action
of $G$ on it is transitive. Isomorphism classes of irreducible $G$-sets are in one-to-one correspondence with
conjugacy classes of subgroups of $G$: to the conjugacy class containing a subgroup $H\subset G$ one associates the isomorphism class $[G/H]$
of the $G$-set $G/H$. The {\em Grothendieck ring} $\Kring{G}$ {\em of finite $G$-sets}
(also called the {\em Burnside ring} of $G$: see, e.g.,
\cite{Knutson}) is the (abelian) group generated by the isomorphism classes of finite $G$-sets
modulo the relation $[A\amalg B]=[A]+[B]$ for finite $G$-sets $A$ and $B$. The multiplication
in $\Kring{G}$ is defined by the 
cartesian product. As an abelian group, $\Kring{G}$
is freely generated by the isomorphism classes of irreducible $G$-sets. The element $1$ in
the ring $\Kring{G}$ is represented by the $G$-set consisting of one point (with the trivial $G$-action).

There is a natural homomorphism from the Burnside ring $\Kring{G}$ to the ring $R(G)$ of representations of the group $G$ which sends a $G$-set $X$ to the (vector) space of functions on $X$. If $G$ is cyclic, then this homomorphism is injective. In general, it is neither injective nor surjective.

For a subgroup $H\subset G$ there are natural maps
$\mbox{Res}_{H}^{G}: \Kring{G}\to \Kring{H}$ and $\mbox{Ind}_{H}^{G}: \Kring{H}\to \Kring{G}$.
The {\em restriction map} $\mbox{Res}_{H}^{G}$ sends a $G$-set X to the same set considered with the $H$-action.
The {\em induction map} $\mbox{Ind}_{H}^{G}$ sends an $H$-set $X$ to the product $G\times X$ factorized
by the natural equivalence: $(g_1, x_1)\equiv (g_2, x_2)$ if there exists $g\in H$ such that
$g_2=g_1g$, $x_2=g^{-1}x_1$ with the natural (left) $G$-action. Both maps are group homomorphisms, however the induction map $\mbox{Ind}_{H}^{G}$ is not a ring homomorphism. 

For an action of a group $G$ on a set $X$ and for a point $x\in X$, let $G_x=\{g\in G: gx=x\}$ be the isotropy group of the point $x$. For a subgroup $H\subset G$ let $X^{(H)}=\{x\in X: G_x=H\}$ be the set of points with the isotropy group $H$.

The Saito duality is applied to the monodromy zeta functions of quasihomogeneous (hypersurface) singularities and thus we
shall restrict our discussion to this situation as well.

Let $f(x_1, \ldots, x_n)$ be a quasihomogeneous polynomial in $n$ variables with reduced weights $\overline{w}_1$, \dots, 
$\overline{w}_n$ and quasidegree $\overline{d}$.
The monodromy transformation of $f$ can be defined as the element $h=h_f\in G_f$ given by
$$
h=\left(\exp(2\pi i \,\overline{w}_1/\overline{d}), \ldots,
\exp(2\pi i \,\overline{w}_n/\overline{d})\right)\,.
$$
As a map from the Milnor fibre $V_f=f^{-1}(1)$ of $f$ to itself,
$h$ defines an action (a faithful one) of the cyclic group $G=\ZZ_{\overline{d}}$ of order $\overline{d}$
on $V_f$. Let
\begin{equation} \label{Defzeta}
\zeta_f(t)=\prod\limits_{q\ge 0} \left(\det(\mbox{id}-t\cdot h_{*}\mbox{\raisebox{-0.5ex}{$\vert$}}{}_{H_q(V_f)})\right)^{(-1)^q}
\end{equation}
be the (classical) monodromy zeta function of $f$ (that is the zeta function of the transformation $h$). One can show that in the described situation one has
$$
\zeta_f(t)=\prod\limits_{m\vert \overline{d}}(1-t^m)^{s_m},
$$
where $s_m=\chi(V_f^{(\ZZ_{\overline{d}/m})})/m$ are integers. If in (\ref{Defzeta}) one considers the action of $h_{*}$ on the reduced homology groups of $V_f$, one obtains the {\em reduced monodromy zeta function} $\widetilde{\zeta}_f(t) = \zeta_f(t)/(1-t)$.

Let a finite $\ZZ_{\overline{d}}\,$-set set $X$ represent an element $a\in \Kring{\ZZ_{\overline{d}}}$.
One can consider $X$ as a (discrete) topological space with a transformation $h$ of order
$\overline{d}$ ($h$ is a generator of the group $\ZZ_{\overline{d}}$). Let $\zeta_a(t)$ be the zeta function
of the transformation $h:X\to X$. The correspondence $a\mapsto \zeta_a(t)$ (being appropriately extended: $\zeta_{a-b}(t):=\zeta_a(t)/\zeta_b(t)$)
defines a map from the Burnside ring $\Kring{\ZZ_{\overline{d}}}$ to the set of functions of
the form~(\ref{phi}). One can easily see that this is a one-to-one correspondence. A function
$\phi$ of the form~(\ref{phi}) corresponds to the element
$\sum\limits_{m\vert \overline{d}}s_m[\ZZ_{\overline{d}}/\ZZ_{\overline{d}/m}]$.
Thus the zeta function of a transformation of order $\overline{d}$ (of a ``good'' topological
space, not necessarily of a finite one, say, of the monodromy transformation $h_f$ above)
can be regarded as an element of the Burnside ring $\Kring{\ZZ_{\overline{d}}}$.

Now let $G$ be a subgroup of the symmetry group $G_f$ of $f$ containing the monodromy transformation $h$.
The description above inspires the following definition.

\begin{definition}
The {\em $G$-equivariant zeta function} of $f$ is the element
\begin{equation}\label{G-zeta}
\zeta_f^G=\sum_{H\subset G} \chi(V_f^{(H)}/G)[G/H]
\end{equation}
of the Burnside ring $\Kring{G}$. 
\end{definition}

The coefficient $\chi(V_f^{(H)}/G)$ is the Euler characteristic of the space (a manifold) of orbits of type $G/H$ in $V_f$.

\begin{definition}
The {\em reduced $G$-equivariant zeta function} of $f$ is $\widetilde{\zeta}_f^G = \zeta_f^G-1$.
\end{definition}

\begin{remarks}
{\bf 1.} For a group of symmetries of a function $f$ ($f$ is not necessarily quasihomogeneous
and $G$ is not necessarily abelian or containing $h$) the element defined by (\ref{G-zeta})
can be regarded as an equivariant Euler characteristic of the Milnor fibre $V_f$.
(This definition was already used in, e.g., \cite{Luck}, \cite{GLM}.)
In particular, under the natural map from the Burnside ring $\Kring{G}$ to the ring of representations $R(G)$, it maps to the equivariant
Euler characteristic of the Milnor fibre 
in the sense of \cite{Wall}. In the situation when the group $G$ contains the monodromy transformation $h$, this element can be regarded as an equivariant version of the monodromy zeta function as well. In particular it determines the classical zeta function $\zeta_f(t)$ of $f$: see Remark~2 below. For a non-degenerate $f$, the analogue of the reduced $G$-equivariant zeta function can be regarded as a $G$-equivariant Milnor number.

{\bf 2.} For a subgroup $H\subset G$ ($G$ is a group of symmetries of $f$) containing
the monodromy transformation $h$, the
$H$-equivariant zeta function of $f$ is the restriction of its $G$-equivariant zeta
function: $\zeta_f^{H}= \mbox{Res}_{H}^{G} \zeta_f^{G}$. In particular, the $G$-equivariant zeta function of $f$ determines the $\langle h\rangle$-equivariant zeta function corresponding to the (cyclic) group generated by the monodromy transformation $h$ of $f$ and therefore the classical zeta function $\zeta_f(t)$ of $f$.
\end{remarks}

\section{Equivariant Saito duality}
The Saito duality is a duality between rational functions of the form (\ref{phi}). As described above, these functions are in one-to-one correspondence with the elements of the Burnside ring
$\Kring{\ZZ_d}$. Let us describe the (classical) Saito duality in terms of the Burnside ring.

A finite $\ZZ_d\,$-set is the union of $\ZZ_d\,$-orbits of the form $\ZZ_d/\ZZ_{d/m}$
(consisting of $m$ points). One can see that the Saito duality is induced by the map
from $\Kring{\ZZ_d}$ into itself which substitutes each orbit
(an irreducible $\ZZ_d\,$-set) consisting of $m$ points by an orbit consisting of $d/m$ points.
An orbit consisting of $d/m$ points can be regarded as the cyclic group $\ZZ_{d/m}$.
However it should not be identified with the isotropy subgroup of the $\ZZ_d\,$-action
on the initial orbit (also isomorphic to $\ZZ_{d/m}$): there is no natural (non-trivial) action
of the group $\ZZ_d$ on its subgroup $\ZZ_{d/m}$. Instead of that one can regard the
(classical) Saito duality as an isomorphism between (the abelian groups)
$\Kring{\ZZ_d}$ and $\Kring{\ZZ_d^*}$ where
$\ZZ_d^*=\mbox{Hom\,}(\ZZ_d,\CC^*)$ is the group of characters of $\ZZ_d$ ($\ZZ_d^*$ is isomorphic
to $\ZZ_d$, but this isomorphism is not canonical). 
An element $a\in \Kring{\ZZ_d}$ can be written as
$$
\sum\limits_{H\subset \ZZ_d} s_H[\ZZ_d/H]\,.
$$
The classical Saito duality associates to $a$ the element
$$
\widehat{a}=\sum\limits_{H\subset \ZZ_d} s_H[\ZZ_d^*/\widetilde{H}]
$$
of the Burnside ring $\Kring{\ZZ_d^*}$, where $\widetilde{H}$ is the dual subgroup of $\ZZ_d^*$.

This inspires the following definition.

\begin{definition}
Let $G$ be a finite abelian group. The {\em equivariant Saito duality} corresponding to the group $G$
(or {\em $G$-Saito duality}) is the group homomorphism
$D_G:\Kring{G}\to \Kring{G^*}$ sending 
an element $a=\sum\limits_{H\subset G} s_H[G/H]$ to the element 
$\widehat{a}=D_Ga=\sum\limits_{H\subset G} s_H[G^*/\widetilde{H}]$, where $\widetilde{H}$ is
the dual to $H$ with respect to $G$.
\end{definition}

One can easily see that $D_G$ is an isomorphism of abelian groups.

\begin{remark}
One can regard the correspondence $a\mapsto \widehat{a}$ as a Fourier transformation from
$\Kring{G}$ to $\Kring{G^*}$.
The understanding of the Saito duality as a duality between objects corresponding to
a group $G$ and objects corresponding to the group $G^*$ is consistent with the idea
that a duality between orbifold Landau-Ginzburg models includes substitution of a
group by the group of its characters: see, e.g., \cite{BH2}, \cite{ET2}.
Let us recall that the symmetry group $G_{\widetilde{f}}$ of the Berglund-H\"ubsch transpose $\widetilde{f}$
of an invertible polynomial $f$ is isomorphic to the group of characters of $G_f$.
\end{remark} 

\section{Equivariant monodromy zeta functions of dual invertible polynomials}
Let $f(x_1, \ldots, x_n)=\sum\limits_{i=1}^n \prod\limits_{j=1}^n x_j^{E_{ij}}$ be an
invertible polynomial, let $\widetilde{f}$ be the Berglund-H\"ubsch transpose of it, and let $G=G_f$ and
$G^*=G_{\widetilde{f}}$ be their symmetry groups. 

\begin{theorem}\label{theo}
The reduced equivariant zeta functions $\widetilde{\zeta}_f^{\,G}$ and $\widetilde{\zeta}_{\widetilde{f}}^{G^*}$
of the polynomials $f$ and $\widetilde{f}$ respectively are $($up to the sign $(-1)^n$$)$ Saito dual to each other:
$$
\widetilde{\zeta}_{\widetilde{f}}^{G^*} = 
(-1)^n D_{G}\widetilde{\zeta}_f^{G}\,.
$$ 
\end{theorem}

\begin{proof}
For a subset $I\subset I_0=\{1,2,\ldots, n\}$, let
$\CC^I:= \{(x_1, \ldots, x_n)\in \CC^n: x_i=0 \mbox{ for }i\notin I\}$,
$(\CC^*)^I:= \{(x_1, \ldots, x_n)\in \CC^n: x_i\ne 0 \mbox{ for }i\in I, x_i=0 \mbox{ for }i\notin I\}$.
One has $\CC^n=\coprod\limits_{I\subset I_0}(\CC^*)^I$, $V_f=\coprod\limits_{I\subset I_0}V_f\cap(\CC^*)^I$. Let $G_f^I\subset G$ and $G_{\widetilde{f}}^I\subset G^*$ be the isotropy subgroups of the actions of $G$ and $G^*$ on the torus $(\CC^*)^I$ respectively. (These isotropy subgroups are the same for all points of $(\CC^*)^I$.)

Let $\ZZ^n$ be the lattice of monomials in the variables $x_1$, \dots, $x_n$ ($(k_1, \ldots, k_n)\in \ZZ^n$ corresponds to the monomial $x_1^{k_1}\cdots x_n^{k_n}$)
and let $\ZZ^I:=\{(k_1, \ldots, k_n)\in \ZZ^n: k_i=0 \mbox{ for }i\notin I\}$.
For a polynomial $g$ in the variables $x_1$, \dots, $x_n$, let $\mbox{supp\,} g\subset \ZZ^n$ be the set of monomials (with non-zero coefficients) in $g$.

One has
\begin{equation}\label{zetas}
\zeta_f^G=\sum\limits_{I\subset I_0}\zeta_f^{G,I} \ \  \mbox{where} \ \ 
\zeta_f^{G,I}:=\chi((V_f\cap (\CC^*)^I)/G)[G/G_f^I]\,,\quad \widetilde{\zeta}_f^G=\zeta_f^G-1\,.
\end{equation}
(Note that $\zeta_f^{G,\emptyset} =0$.) Here $\chi((V_f\cap (\CC^*)^I)/G)=\chi(V_f\cap (\CC^*)^I)\vert G_f^I\vert/\vert G\vert$. 
One can consider only the summands $\zeta_f^{G,I}$ with the coefficient $\chi((V_f\cap (\CC^*)^I)/G)$ different from zero.
The coefficient $\chi((V_f\cap (\CC^*)^I)/G)$
is different from $0$ if and only if $\chi(V_f\cap (\CC^*)^I)\ne 0$. From the Varchenko formula
\cite{Varch}, it follows that the latter Euler characteristic is different from zero
if and only if $\mbox{supp\,} f\cap \ZZ^I$ consists of $\vert I\vert$ points. 

Let $\vert\mbox{supp\,} f \cap \ZZ^I\vert=\vert I\vert =:k$. Renumbering the coordinates $x_i$ and 
the monomials in $f$ permits to assume that $I=\{ 1, \ldots , k\}$ and the matrix $E$ is of the form
$$E = \left( \begin{array}{cc} E_I & 0 \\ \ast & E_{\overline{I}} \end{array} \right),
$$ 
where $E_I$ and $E_{\overline{I}}$ are square matrices of sizes $k \times k$ and $(n-k) \times (n-k)$ respectively. One has $\vert\mbox{supp\,} \widetilde{f} \cap \ZZ^{\overline{I}}\vert=\vert {\overline{I}}\vert=n-k$.

\begin{lemma} Under the above assumptions, one has 
$$ G_{\widetilde{f}}^{\overline{I}} = \widetilde{G_f^I}.$$
\end{lemma}

\begin{sloppypar}

\begin{proof} Let $g$ and $\widetilde{g}$ be the restrictions of the polynomials $f$ and $\widetilde{f}$  to $\CC^I$ and to $\CC^{\overline{I}}$ respectively. There are the following two exact sequences
\begin{eqnarray*}
& & 0\longrightarrow G_f^I \longrightarrow G_f \longrightarrow G_g\, , \\
& & 0\longrightarrow G_{\widetilde{f}}^{\overline{I}} \longrightarrow G_{\widetilde{f}} \longrightarrow G_{\widetilde{g}}\, .
\end{eqnarray*}
One has $|G_f|=|G_{\widetilde{f}}|= d = \det E$, $|G_g|=\det E_I$, $|G_{\widetilde{g}}|= \det E_{\overline{I}}$, and therefore $|G_g||G_{\widetilde{g}}|=d$. This implies that 
\begin{equation} \label{order}
|G_f^I| |G_{\widetilde{f}}^{\overline{I}}| \geq d.
\end{equation}
(Moreover, $|G_f^I| |G_{\widetilde{f}}^{\overline{I}}| = d$ if and only if the homomorphisms $G_f \to G_g$ and $G_{\widetilde{f}} \to G_{\widetilde{g}}$ are surjective.)

Elements $\underline{\lambda} = (\exp(2\pi i\,\alpha_1), \ldots, \exp(2\pi i\,\alpha_n)) \in G_{\widetilde{f}}$ and $\underline{\mu} = (\exp(2\pi i\,\beta_1), \ldots, \exp(2\pi i\,\beta_n)) \in G_f$ belong to $G_{\widetilde{f}}^{\overline{I}}$ and to $G_f^I$ if and only if $\alpha_{k+1} \equiv \ldots \equiv \alpha_n \equiv 0 \, \mbox{mod}\, 1$ and $\beta_1 \equiv \ldots \equiv \beta_k \equiv 0 \, \mbox{mod}\, 1$ respectively. Therefore
$$(\underline{\alpha},\underline{\beta})_E=(\alpha_1,  \ldots , \alpha_k  , 0 , \ldots , 0)\left( \begin{array}{cc} E_I & 0 \\ \ast & E_{\overline{I}} \end{array} \right) \left( \begin{array}{c} 0 \\ \vdots \\ 0 \\ \beta_{k+1} \\ \vdots \\ \beta_n \end{array} \right) =0
$$
and $\langle\underline{\lambda}, \underline{\mu}\rangle_E=1$. Thus $G_{\widetilde{f}}^{\overline{I}} \subset \widetilde{G_f^I}$ and therefore 
\begin{equation} \label{order2}
|G_f^I| |G_{\widetilde{f}}^{\overline{I}}| \leq d.
\end{equation}
The inequalities (\ref{order})  and (\ref{order2}) imply that $|G_f^I| |G_{\widetilde{f}}^{\overline{I}}|=d$, $|G_f^I | = \det E_{\overline{I}}$, $G_{\widetilde{f}}^{\overline{I}}= \det E_I$, and $G_{\widetilde{f}}^{\overline{I}} = \widetilde{G_f^I}$.
\end{proof}

\end{sloppypar}

Due to the Varchenko formula, the Euler characteristic of the intersection $V_f\cap (\CC^*)^I$ of
the Milnor fibre $V_f$ with the torus $(\CC^*)^I$ is equal to $(-1)^{k-1} \det E_I = (-1)^{k-1} |G_f|/|G_f^I|$. Therefore 
$$
\chi((V_f\cap (\CC^*)^I)/G)=(-1)^{\vert I\vert-1}\,.
$$
In the same way
$$
\chi((V_{\widetilde{f}}\cap (\CC^*)^{\overline{I}})/G^*)=(-1)^{n-\vert I\vert-1}\,.
$$

This establishes the equivariant Saito duality (up to the sign $(-1)^n$)
between all the summands in the equation (\ref{zetas}) for ${\zeta}_f^G$ and
${\zeta}_{\widetilde{f}}^{G^*}$
corresponding to proper subsets $I\subsetneq I_0$. One can see that the summand in ${\zeta}_f^G$
corresponding to $I=I_0$ is dual (up to the sign $(-1)^{n}$)
to the summand $-1$ in the reduced zeta function $\widetilde{\zeta}_{\widetilde{f}}^{G^*}$.
This implies the statement. 
\end{proof}

\section{Geometric roots of the monodromy and the equivariant duality}
In \cite{MMJ} there were defined geometric roots of the monodromy transformation.
For an invertible polynomial $f$ a {\em geometric root} of degree
$c_f=\gcd(w_1, \ldots, w_n)$ ($w_i$ are the canonical weights of $f$) of the monodromy
transformation $h_f$ is an element $\sqrt{h}=\sqrt{h}_f$
of the symmetry group $G_f$ such that $({\sqrt{h}})^{c_f}=h_f$.
The order of the monodromy transformation $h_f$ of $f$ is equal to the
reduced degree $\overline{d}=d/c_f$ of the polynomial $f$. This implies the following statement.

\begin{proposition}
Geometric roots of degree $c_f$ of the monodromy transformation $h_f$ exist if and only if
the symmetry group $G_f$ of $f$ is cyclic.
\end{proposition}

If geometric roots of degree $c_f$ of the monodromy transformation $h_f$ exist,
then each of them is a generator of the symmetry group $G_f\cong\ZZ_d$.
In this case the symmetry group $G_{\widetilde{f}}\cong G_f^*$ of the Berglund--H\"ubsch transpose $\widetilde{f}$ of $f$ is also a cyclic group of order $d$.
The monodromy transformation $h_{\widetilde{f}}$ is an element of order $\overline{d}_{\widetilde{f}}=d/c_{\widetilde{f}}$ in it and
therefore it has a geometric root of order $c_{\widetilde{f}}$. Together with the fact that the equivariant Saito duality for the group $\ZZ_d$
differs from the classical one only by the sign $(-1)$, Theorem \ref{theo} implies the following statement.

\begin{corollary} 
If geometric roots $\sqrt{h}_f$ of degree $c_f$ of the monodromy transformation $h_f$ exist, then geometric roots
$\sqrt{h}_{\widetilde{f}}$ of degree $c_{\widetilde{f}}$ of the monodromy transformation $h_{\widetilde{f}}$ also exist and one has
$$
\widetilde{\zeta}_{\sqrt{h}_{\widetilde{f}}}(t)=
\left(\widetilde{\zeta}_{\sqrt{h}_f}^*(t)\right)^{(-1)^{n-1}}
$$
with respect to the quasidegree $d=\det E$.
\end{corollary} 

This statement was proved in \cite{MMJ} for non-degenerate invertible polynomials in 3 variables and
for such polynomials in an arbitrary number
of variables of pure loop or chain type. 


\bigskip
\noindent Leibniz Universit\"{a}t Hannover, Institut f\"{u}r Algebraische Geometrie,\\
Postfach 6009, D-30060 Hannover, Germany \\
E-mail: ebeling@math.uni-hannover.de\\

\medskip
\noindent Moscow State University, Faculty of Mechanics and Mathematics,\\
Moscow, GSP-1, 119991, Russia\\
E-mail: sabir@mccme.ru

\end{document}